\documentclass{amsart}
\usepackage{amsfonts,amssymb,amsmath,amsthm}
\usepackage{url}
\usepackage{enumerate}

\theoremstyle{plain}

\newtheorem{thm}{Theorem}[section]
\newtheorem*{thm*}{Theorem}

\newtheorem*{lem*}{Lemma}

\theoremstyle{remark}

\newtheorem{rem}[thm]{Remark}
\newtheorem*{rem*}{Remark}

\theoremstyle{definition}

\newtheorem*{notation*}{Notation}

\newtheorem{defn}[thm]{Definition}
\newtheorem*{defn*}{Definition}

\newtheorem*{example*}{Example}

\newtheorem*{examples*}{Examples}

\newtheorem*{problem*}{Problem}

\newtheorem*{xca*}{Exercise}

\newtheorem*{xcas*}{Exercises}

\newtheorem*{condition*}{Condition}

\makeatother

\author{Adam Parusi\'{n}ski}
\address{
Laboratoire J.-A. Dieudonn{\'e}\\
Universit{\'e} de Nice Sophia Antipolis\\
Parc Valrose\\
06108 Nice Cedex 02\\
France\\
}
\email{parus@unice.fr}
\author{Jean-Philippe Rolin}
\address{
Universit{\'e} de Bourgogne\\
UFR Sciences et Techniques\\
Facult{\'e} des Sciences Mirandes\\
9 avenue Alain Savary - BP 47870\\
21078 Dijon Cedex\\
France
}
\email{rolin@u-bourgogne.fr}
\thanks{}


\begin{document}

\title[Weierstrass Preparation]{Note on the Weierstrass Preparation Theorem in Quasianalytic Local
Rings}

\begin{abstract}
Consider quasianalytic local rings of germs of smooth functions closed
under composition, implicit equation, and monomial division. We show
that if the Weierstrass Preparation Theorem holds in such a ring then
all elements of it are germs of analytic functions. 
\end{abstract}
\maketitle

\section{Introduction and main results}
Since the original work of Borel \cite{borel:prolongement_ana,borel:series_fractions_rationnelles},
the notion of quasianalytic rings of infinitely differentiable functions
has been studied intensively (see for example the expositary article
on quasianalytic local rings written by V. Thilliez \cite{thilliez:quasianlytic_local_rings}).
Recall that a ring $\mathcal{C}_{n}$ of smooth germs at the origin
of $\mathbb{R}^{n}$ is called \emph{quasianalytic} if the only element
of $\mathcal{C}_{n}$ which admits a zero Taylor expansion is the
zero germ.

The early works of Denjoy \cite{denjoy:quasianalytiques} and Carleman
\cite{carleman:quasianalytiques} show a deep connection between the
growth of partial derivatives of $\mathcal{C}^{\infty}$ germs at
the origin and the quasianalyticity property, leading to the notion
of \emph{quasianalytic Denjoy-Carleman classes of functions}. The
algebraic properties of such rings, namely their stability under several
classical operations, such as composition, differentiation, implicit
function, is well understood (see for example \cite{roumieu:ultra}
and \cite{komatsu:implicit}).

These stability properties have allowed a study of quasianalytic classes
from the point of view of \emph{real analytic geometry}, that is the
investigation of the properties of subsets of the real spaces locally
defined by equalities and inequalities satisfied by elements of these
rings. For example, it is shown in \cite{bm_Resol_DenjoyCarleman}
how the resolution of singularities extends to the quasianalytic framework.

\paragraph{{}}

However, two classical properties, namely Weierstrass division and
Weierstrass preparation, seem to cause trouble in the quasianalytic
setting. For example it has been proved by Childress in \cite{childress:weierstrass_quasianalytic}
that quasianalytic Denjoy-Carleman classes might not satisfy Weierstrass
division. Since Weierstrass preparation is usually introduced as a
consequence of Weierstrass division, it is classically considered
that Weierstrass preparation should fail in a quasianalytic framework.
So far, no explicit counterexample has been given. Moreover, we do
not know any example of a ring of smooth functions for which the Weierstrass
preparation theorem holds, but the Weierstrass division fails.

\paragraph{{}}

We are interested here in what we call in the next section a \emph{quasianalytic
system}, that is a collection of quasianalytic rings of germs of smooth
functions which contain the the analytic germs, and which is closed
under composition, partial differentiation and implicit function.
Such systems have been investigated in several works from the point
of view of real analytic geometry or o-minimality \cite{vdd:speiss:multisum,rsw,bm_Resol_DenjoyCarleman,rss}.
It is worth noticing that, in these papers, the possible failure of
Weierstrass preparation leads to a study mostly based on resolution
of singularities.

\paragraph{{}}

In such a context, a nice result has been obtained by Elkhadiri and
Soufli in \cite{elkhadiri_soufli:weiertsrass_division_quasianalytic}. They prove,
in a remarkably simple way, that if a quasianalytic system satisfies
Weierstrass division, then it coincides with the analytic system:
all its germs are analytic. The proof is based on the following idea.
In order to prove that a given real germ $f$ is analytic at the origin
of $\mathbb{R}^{n}$, they prove that $f$ extends to a holomorphic
germ at the origin of $\mathbb{C}^{n}$. This extension is built by
considering the \emph{complex formal extension} $\hat{f}\left(x+iy\right)\in\mathbb{C}\left[\left[x,y\right]\right]$,
where $\hat{f}$ is the Taylor expansion of $f$ at the origin. The
real and imaginary parts of this series statisfy the Cauchy-Riemann
equations. Moreover, the Weierstrass division of $f\left(x+t\right)$
by the polynomial $t^{2}+y$ shows that these real and imaginary parts
are the Taylor expansions of two germs which belong to the initial
quasianalytic system. By quasianalyticity, these two germs also satisfy
the Cauchy-Riemann equations. They consequently provide the real and
imaginary parts of a holomorphic extension of $f$.

\paragraph{{}}

Our goal is to use Elkhadiri and Soufli's methods to prove that a
quasianalytic system in which Weierstrass preparation holds coincides
with the analytic system. This result apply in particular to the examples
of quasianalytic systems mentionned above. We still don't know if
any of the rings they contain (besides the rings of germs in one variable)
is noetherian or not.

\paragraph{{}}

A similar property of failure of Weierstrass preparation has been
announced in \cite{acquistapace_broglia_nicoara:failure_weirstrass_preparation-1}
for the quasianalytic Denjoy-Carleman classes. The approach there,
pretty different than ours, leads to a precise investigation of the
following \emph{extension problem}: does a function belonging to a
quasianalytic Denjoy-Carleman class defined on the positive real axis
extends to a function belonging to a wider Denjoy-Carleman class defined
on the real axis? The authors actually produce an explicit example
of non-extendable function with additional properties which permit
contradicting Weierstrass preparation.\

\section{Notations and main result.}

\begin{notation*} For $n\in\mathbb{N}$, we denote by $\mathcal{E}_{n}$
the ring of smooth germs at the origin of $\mathbb{R}^{n}$ and by
$\mathcal{A}_{n}\subset\mathcal{E}_{n}$ the subring of analytic germs.

For every $f\in\mathcal{E}_{n}$ , we denote by $\hat{f}\in\mathbb{R}\left[\left[x_{1},\ldots,x_{n}\right]\right]$
its (infinite) Taylor expansion at the origin.

Finally, we denote $\left(x_{1},\ldots,x_{n}\right)$ by $\boldsymbol{x}$
and $\left(x_{1},\ldots,x_{n-1}\right)$ by $\boldsymbol{x}'$.
\end{notation*}

\begin{defn} \label{def:quasianalytic_system}Consider a collection
$\mathcal{C}=\left\{ \mathcal{C}_{n},n\in\mathbb{N}\right\} $ of
rings of germs of smooth functions at the origin of $\mathbb{R}^{n}$.
We say that $\mathcal{C}$ is a \emph{quasianalytic system} if the
following properties hold for all $n\in\mathbb{N}$: 
\begin{enumerate}
\item The algebra $\mathcal{A}_{n}$ is contained is $\mathcal{C}_{n}$. 
\item (Stability by composition) If $f\in\mathcal{C}_{n}$ and $g_{1},\ldots,g_{n}\in\mathcal{C}_{m}$
with $g_{1}\left(0\right)=\cdots=g_{n}\left(0\right)=0$, then $f\left(g_{1},\ldots,g_{n}\right)\in\mathcal{C}_{m}$. 
\item (Stability by implicit equation, assuming $n>0$) If $f\in\mathcal{C}_{n}$
satisfies $f\left(0\right)=0$ and $\left(\partial f/\partial x_{n}\right)\left(0\right)\ne0$,
then there exists $\varphi\in\mathcal{C}_{n-1}$ such that $\varphi\left(0\right)=0$
and $f\left(\boldsymbol{x}',\varphi\left(\boldsymbol{x}'\right)\right)=0$. 
\item (Stability under monomial division) If $f\in\mathcal{C}_{n}$ satisfies
$f\left(\boldsymbol{x}',0\right)=0$, then there exists $g\in\mathcal{C}_{n}$
such that $f\left(\boldsymbol{x}\right)=x_{n}g\left(\boldsymbol{x}\right)$. 
\item (Quasianalyticity) For every $n\in\mathbb{N}$, the Taylor map $f\mapsto\hat{f}$
is injective on $\mathcal{C}_{n}$. 
\end{enumerate}
\end{defn} \begin{rem} It can easily be seen that the above properties
imply that the algebras $\mathcal{C}_{n}$ are closed under partial
differentiation (see \cite[p.423]{rss} for example).\end{rem} \begin{defn}
A germ $f\in\mathcal{E}_{n}$ is \emph{of order $n$ in the variable
$x_{n}$} if $f\left(\boldsymbol{0},x_{n}\right)=x_{n}^{d}u\left(x_{n}\right)$,
where $u\left(0\right)\ne0$ (that is, $u$ is a \emph{unit} of $\mathcal{E}_{1}$).
\end{defn}

\begin{defn} We say that a quasianalytic system $\mathcal{C}$ \emph{satisfies
Weierstrass preparation} if, for all $n\in\mathbb{N}$, the following
statement $\left(\mathcal{W}_{n}\right)$ holds : every $f\in\mathcal{C}_{n}$
of order $d$ in the variable $x_{n}$ can be written 
\[
f=U\left(\boldsymbol{x}\right)\left(x_{n}^{d}+a_{1}\left(\boldsymbol{x}'\right)x_{n}^{d-1}+\cdots+a_{d}\left(\boldsymbol{x}'\right)\right),
\]
 where $U\in\mathcal{C}_{n}$, $a_{1},\ldots,a_{d}\in\mathcal{C}_{n-1}$,
$U\left(0\right)\ne0\text{ and }a_{1}\left(0\right)=\cdots=a_{d}\left(0\right)=0.$
\end{defn} Our main result is the following : \begin{thm*} If the
quasianalytic system $\mathcal{C}$ satisfies Weierstrass Preparation,
then it coincides with the analytic system: for all $n\in\mathbb{N}$,
$\mathcal{C}_{n}=\mathcal{A}_{n}$.\end{thm*} \begin{rem} We will
actually prove that the conclusion of the theorem is true once $\mathcal{W}_{3}$
holds. \end{rem}

\section{Proof of the theorem}

We consider in this section a quasianalytic system $\mathcal{C}$
which satisfies Weierstrass Preparation.

\paragraph{{}}

In order to prove the theorem, it is enough to prove that $\mathcal{C}_{1}=\mathcal{A}_{1}$.
In fact, it is noticed in \cite{elkhadiri_soufli:weiertsrass_division_quasianalytic}
that the equality $\mathcal{C}_{1}=\mathcal{A}_{1}$ implies $\mathcal{C}_{n}=\mathcal{A}_{n}$
for all $n\in\mathbb{N}$. The argument is the following. If $f\in\mathcal{C}_{n}$
(and $n>1$) then, for every $\xi\in\mathbb{S}^{n-1}$, the germ $f_{\xi}\colon t\mapsto f\left(t\xi\right)$
belongs to $\mathcal{\mathcal{C}}_{1}$. Hence, under the assumption
$\mathcal{C}_{1}=\mathcal{A}_{1}$, the germ $f_{\xi}$ is analytic.
Thanks to a result of \cite{bochnak_siciak:analytic_functions_topological_vector_spaces},
this implies that $f\in\mathcal{A}_{n}$. \begin{lem*} Let $f\in\mathcal{C}_{n}$
such that $f\left(\boldsymbol{0},x_{n}\right)=x_{n}^{2}+x_{n}^{3}+h\left(x_{n}\right)$,
where $h\in\mathcal{C}_{1}$ has order greater than $3$. Then there
exists $f_{0},f_{1}\in\mathcal{C}_{n}$ such that 
\[
f\left(\boldsymbol{x}\right)=f_{0}\left(\boldsymbol{x}',x_{n}^{2}\right)+x_{n}f_{1}\left(\boldsymbol{x}',x_{n}^{2}\right).
\]
 \end{lem*} \begin{proof} We introduce the germs $g_{0}\colon\boldsymbol{x}\mapsto\left(f\left(\boldsymbol{x}',x_{n}\right)+f\left(\boldsymbol{x}',-x_{n}\right)\right)/2$$ $
and $g_{1}\colon\boldsymbol{x}\mapsto\left(f\left(\boldsymbol{x}',x_{n}\right)-f\left(\boldsymbol{x}',-x_{n}\right)\right)/2$,
which both belong to $\mathcal{A}_{n}$ and satisfy $f=g_{0}+g_{1}$.
They are respectively even and odd in the variable $x_{n}$. Hence
the exponents of $x_{n}$ in their Taylor expansions at the origin
are respectively even and odd.

The order of $g_{0}$ in the variable $x_{n}$ is exactly $2$, so
is the order in $x_{n}$ of the germ $F\colon\left(\boldsymbol{x},t\right)\mapsto g_{0}\left(\boldsymbol{x}\right)-t$,
which belongs to $\mathcal{C}_{n+1}$. Since the system $\mathcal{C}$
satisfies Weierstrass preparation, there exist $\varphi_{1},\varphi_{2}\in\mathcal{C}_{n}$
and a unit $U\in\mathcal{C}_{n+1}$ such that $\varphi_{1}\left(\boldsymbol{0}\right)=\varphi_{2}\left(\boldsymbol{0}\right)=0$
and 
\[
F\left(\boldsymbol{x},t\right)=\left(x_{n}^{2}+\varphi_{1}\left(\boldsymbol{x}',t\right)x_{n}+\varphi_{0}\left(\boldsymbol{x}',t\right)\right)\cdot U\left(\boldsymbol{x},t\right).
\]
 We claim that $\varphi_{1}=0$. In fact, considering the Taylor expansions,
we have:

\[
\widehat{F}\left(\boldsymbol{x},t\right)=\left(x_{n}^{2}+\hat{\varphi}_{1}\left(\boldsymbol{x}',t\right)x_{n}+\hat{\varphi}_{0}\left(\boldsymbol{x}',t\right)\right)\cdot\widehat{U}\left(\boldsymbol{x},t\right).
\]
 Now it stems from the classical proof of Weierstrass preparation
theorem for formal series that the support of $x_{n}^{2}+\hat{\varphi}_{1}\left(\boldsymbol{x}',t\right)x_{n}+\hat{\varphi}_{0}\left(\boldsymbol{x}',t\right)$
is contained in the sub-semigroup of $\mathbb{N}^{n+1}$ generated
by the support of $\widehat{F}$. Hence this support contains only
even powers of the variable $x_{n}$, and $\hat{\varphi}_{1}=0$.
Since the system $\mathcal{C}$ is quasianalytic (point 5 of Definition
\ref{def:quasianalytic_system}), $\varphi_{1}=0$.

Notice that the order of the germ $\left(\boldsymbol{x}',z,t\right)\mapsto z+\varphi_{0}\left(\boldsymbol{x}',t\right)$
in the variable $t$ is $1$. Since the system $\mathcal{C}$ is closed
under implicit equation (point $3$ of Definition \ref{def:quasianalytic_system}),
there exists a germ $f_{0}\in\mathcal{C}_{n}$ such that 
\[
z+\varphi_{0}\left(\boldsymbol{x}',t\right)=0\Longleftrightarrow t=f_{0}\left(\boldsymbol{x}',z\right).
\]
 We deduce that 
\begin{align*}
t=g_{0}\left(\boldsymbol{x}\right)\Longleftrightarrow F\left(\boldsymbol{x},t\right)=0 & \Longleftrightarrow x_{n}^{2}+\varphi_{0}\left(\boldsymbol{x}',t\right)=0\\
 & \Longleftrightarrow t=f_{0}\left(\boldsymbol{x}',x_{n}^{2}\right),
\end{align*}
 that is $g_{0}\left(\boldsymbol{x}\right)=f_{0}\left(\boldsymbol{x}',x_{n}^{2}\right)$.

In the same way, we notice that the order of $g_{1}$ in the variable
$x_{n}$ is exactly $3$. Moreover, $g_{1}\left(\boldsymbol{x}',0\right)=0$.
By stability under monomial division (point $4$ of Definition \ref{def:quasianalytic_system})
there exists $\bar{g}_{1}\in\mathcal{C}_{n}$ such that $g_{1}\left(\boldsymbol{x}\right)=x_{n}\bar{g}_{1}\left(\boldsymbol{x}\right)$.
The germ $\bar{g}_{1}$ is even in the variable $x_{n}$ and its order
in this variable is exactly $2$.

Therefore there exists a germ $f_{1}\in\mathcal{C}_{n}$ such that
$\bar{g}_{1}\left(\boldsymbol{x}\right)=f_{1}\left(\boldsymbol{x}',x_{n}^{2}\right)$,
and the lemma is proved. \end{proof}

\begin{rem*}
Given a germ $f\in\mathcal{E}_{n}$, the existence of $f_{0}$ and
$f_{1}$ in $\mathcal{E}_{n}$ which satisfy the statement of the
lemma is a well known fact (see for example \cite[p. 12]{malgrange_ideauxfonctionsdiff}). But the
classical proof, whose first step consists in transforming $f$ in
a flat germ, cannot work in a quasianalytic system. 
\end{rem*}

\begin{proof}[Proof of the Theorem] Consider a germ $h\in\mathcal{C}_{1}$.
Up to adding a polynomial, we may suppose that $h\left(x_{1}\right)=x_{1}^{2}+x_{1}^{3}+\ell\left(x_{1}\right)$,
where the order of $\ell$ in the variable $x_{1}$ is greater than
$3$. We define the germ $f\in\mathcal{C}_{2}$ by $f\colon\left(x_{1},x_{2}\right)\mapsto h\left(x_{1}+x_{2}\right)$.
According to the lemma, there exist two germs $f_{0}$ and $f_{1}$
in $\mathcal{C}_{2}$ such that 
\[
f\colon\left(x_{1},x_{2}\right)\longmapsto f_{0}\left(x_{1},x_{2}^{2}\right)+x_{2}f_{1}\left(x_{1},x_{2}^{2}\right).
\]
 We introduce the complex germ $H$ defined by 
\[
H\colon z=x_{1}+\mathrm{i}x_{2}\in\mathbb{C}\longmapsto f_{0}\left(x_{1},-x_{2}^{2}\right)+\mathrm{i}x_{2}f_{1}\left(x_{1},-x_{2}^{2}\right).
\]
 We see that $H\left(x_{1},0\right)=f\left(x_{1},0\right)=h\left(x_{1}\right)$.
Hence the theorem is proved once we have proved that the germ $H$
is holomorphic, that is that its real and imaginary parts satisfy
the Cauchy-Riemann equations.

Consider the Taylor expansion $\hat{h}\left(x_{1}\right)=\sum_{n\ge0}h_{n}x_{1}^{n}\in\mathbb{R}\left[\left[x_{1}\right]\right]$
of the germ $h$. The real and imaginary parts of the formal series
$\widehat{H}\left(x_{1}+\mathrm{i}x_{2}\right)\in\mathbb{C}\left[\left[x_{1},x_{2}\right]\right]$
defined by 
\[
\widehat{H}\left(x_{1}+\mathrm{i}x_{2}\right)=\sum_{n\ge0}h_{n}\left(x_{1}+\mathrm{i}x_{2}\right)^{n}
\]
 are the series $\hat{f}_{0}\left(x_{1},-x_{2}^{2}\right)$ and $x_{2}\hat{f}_{1}\left(x_{1},-x_{2}^{2}\right)$.
Since the coefficients $h_{n}$, $n\in\mathbb{N}$, are real numbers,
these series satisfy the Cauchy-Riemann equations. By quasianalyticity,
the germs $f_{0}\left(x_{1},-x_{2}^{2}\right)$ and $x_{2}f_{1}\left(x_{1},x_{2}^{2}\right)$
satisfy the same equations.

We deduce that the complex germ $H$ is holomorphic, and thus the
germ $h$ is analytic.\end{proof} \begin{rem} In the proof of the
theorem, the lemma is applied to the germ $f$ which belongs to $\mathcal{C}_{2}$.
Hence the single hypothesis $\mathcal{W}_{3}$ is actually required.
\end{rem}

 \global\long\def\etalchar#1{$^{#1}$}
 \global\long\def\cprime{$'$}
 \providecommand{\bysame}{\leavevmode\hbox to3em{\hrulefill}\thinspace}
\providecommand{\MR}{\relax\ifhmode\unskip\space\fi MR } 
\providecommand{\MRhref}[2]{%
  \href{http://www.ams.org/mathscinet-getitem?mr=#1}{#2}
} \providecommand{\href}[2]{#2}

\end{document}